\documentclass[a4paper,11pt,reqno]{amsart}

\usepackage{amssymb,latexsym,amsthm}
\usepackage{amsmath} 
\usepackage{amsfonts}
\usepackage{graphicx}
\usepackage{color}
\usepackage{enumitem}
\usepackage[numbers,comma,square,sort&compress]{natbib}
\usepackage{hyperref} 
\usepackage[text={5.5in,9.5in},centering]{geometry}

\newtheorem{lemma}{Lemma}[section]

\newtheorem{definition}{Definition}[section]

\newtheorem{proposition}{Proposition}[section]

\newtheorem{theorem}{Theorem}[section]
\newtheorem*{claim*}{Claim}

\title[{\tiny Synchronization in minimal iterated function systems}]{Synchronization in minimal iterated function systems on compact manifolds}

\author[{\tiny{Ale Jan Homburg}}]{Ale Jan Homburg}
\address{KdV Institute for Mathematics, University of Amsterdam, Science park 105-107, 1098 XG Amsterdam, Netherlands}
\address{Department of Mathematics, VU University Amsterdam, De Boelelaan 1081, 1081 HV Amsterdam, Netherlands}
\email{a.j.homburg@uva.nl}

\begin{document}

\begin{abstract}
We treat synchronization for iterated function systems generated by diffeomorphisms on compact manifolds.
Synchronization here means the convergence of orbits starting at different initial conditions when
iterated by the same sequence of diffeomorphisms.
The iterated function systems admit a description as skew 
product systems of diffeomorphisms on compact manifolds  
driven by shift operators. 
Under open conditions including transitivity 
and negative fiber Lyapunov exponents, we prove the existence of a unique 
attracting invariant graph for the skew product system.  
This explains the occurrence of synchronization.  The result extends previous results for iterated function systems by diffeomorphisms on the circle, to arbitrary compact manifolds.


%

\noindent MSC 37C05, 37D30
\end{abstract}

\maketitle

\section{Introduction}\label{s:i}

We consider iterated function systems generated by finitely many diffeomorphisms on a compact manifold.
We thus consider compositions of diffeomorphisms $f_0,\ldots,f_k$ acting on a compact manifold $M$.
For each iterate a diffeomorphism $f_i$ is picked independently with a given probability $p_i$. 
Our focus will lie on the combination of two properties:
\begin{enumerate}[label=(\roman*)]
\item the iterated function systems are minimal, meaning that
for each point there is a sequence of diffeomorphisms that gives a dense orbit in the manifold;
\item
the iterated function systems displays synchronization,
meaning that typically orbits converge to each other when iterated by the same sequence of diffeomorphisms.
\end{enumerate}
We moreover demand robustness where the properties persist under small perturbations of the generating diffeomorphisms 

Minimal iterated function systems on compact manifolds have been constructed before in \citep{homnas,ifs,bbd}. 
Synchronization for minimal iterated function systems on compact manifolds has been established before only for
iterated function systems on the circle \citep{ant84}, see also \citep{MR2358052,1086.37026,kai93,MR2425065,mal17}.

\begin{figure}[htbp]

\begin{center}
 \includegraphics[height=4.5cm]{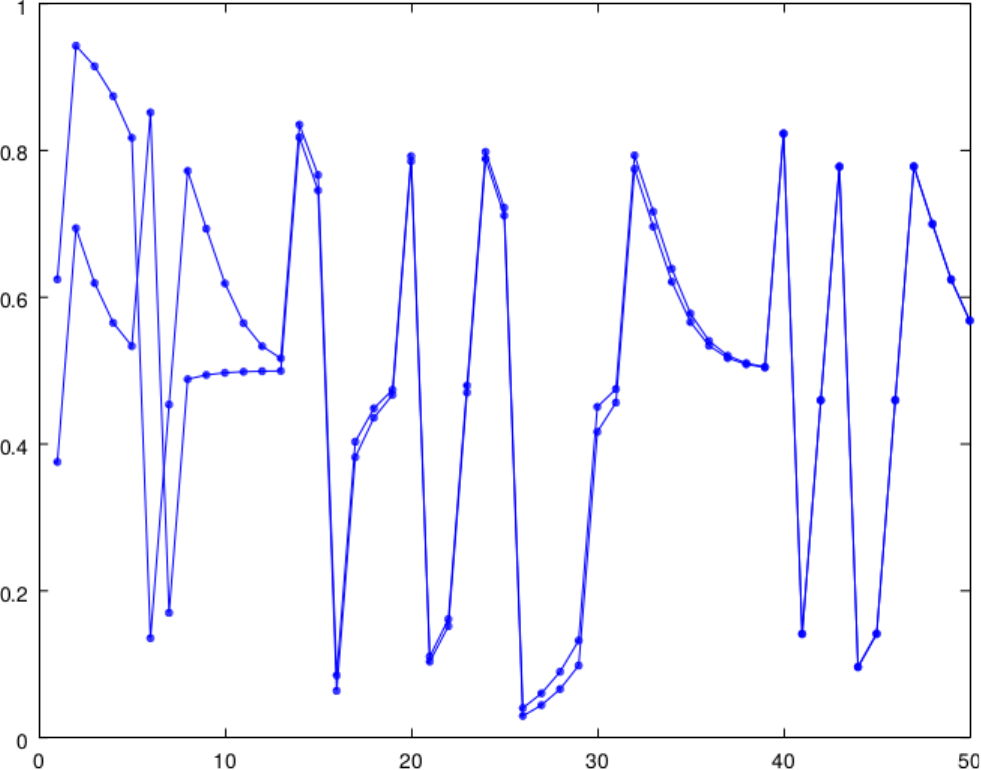}\hspace{0.4cm} \includegraphics[height=4.5cm]{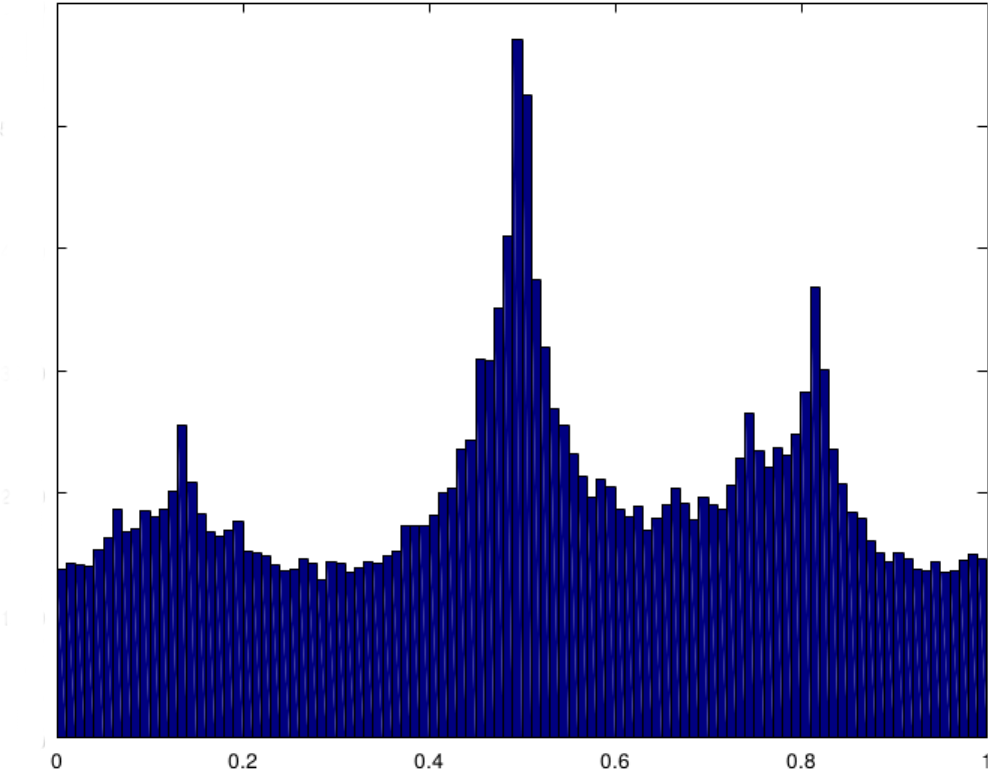} 
\caption{The left frame shows numerically computed time series for the iterated function system generated by
$f_0 (x) = x + \frac{1}{4\pi} \sin (2\pi x)$  (chosen with probability $p_0 = 0.5$)
and $f_1 (x) = x + \frac{1}{\pi} \mod 1$ (chosen with probability $p_1 = 0.5$).
Time series with two different initial conditions but identical compositions are shown, where the points are connected by lines.
The orbits appear to converge to each other.
The right frame shows the histogram of a numerically computed orbit, that appears to lie dense in $[0,1]$.
\label{f:synchonT}}
\end{center}

\end{figure}
For the purpose of illustration we describe an example.
Let $f_0$ be a Morse-Smale diffeomorphism on the circle with a unique attracting and a unique repelling fixed point.
Let $f_1$ be a smooth diffeomorphism on the circle with irrational rotation number, so that its orbits are dense.
Consider the iterated function system generated by $f_0, f_1$, were $f_0,f_1$ are picked independently with positive probabilities
$p_0, p_1 = 1-p_0$. This iterated function system is clearly minimal. It is not difficult to demonstrate that iterated functions 
systems generated by
diffeomorphisms that are $C^1$-small perturbations of $f_0,f_1$ are also minimal.
It has been established that such iterated function systems display synchronization, as illustrated in Figure~\ref{f:synchonT}.

Aim of this paper is to provide constructions of minimal iterated function systems that display synchronization, in a robust way, on any compact manifold. 

Figure~\ref{f:synchonT2} illustrates a two dimensional example: it shows synchronization of a minimal iterated function system on the two dimensional torus. 
At the end of this paper in Proposition~\ref{p:example} we  establish  minimality and synchronization of this iterated function system 
and of $C^1$-small perturbations of it. Analogous examples can be given for iterated function systems on higher dimensional tori.
%

\begin{figure}[htbp]

\begin{center}
 \includegraphics[height=3.2cm]{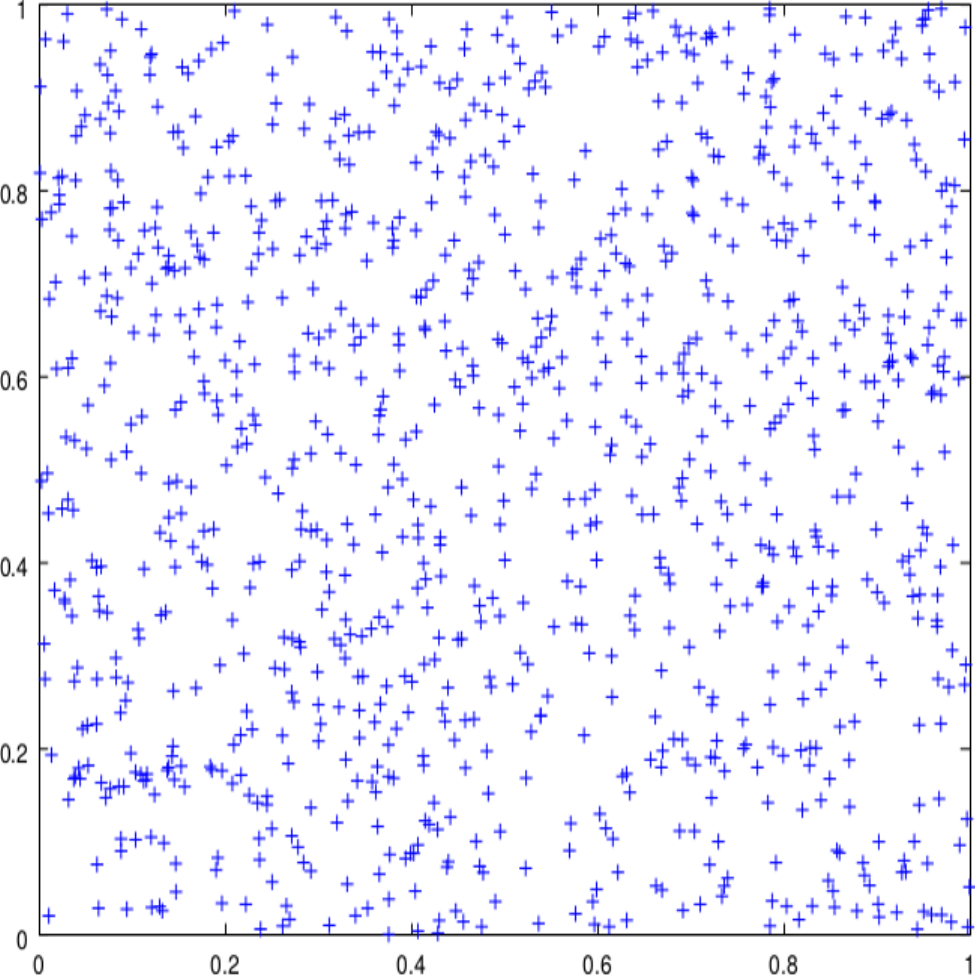}\hspace{0.2cm} \includegraphics[height=3.2cm]{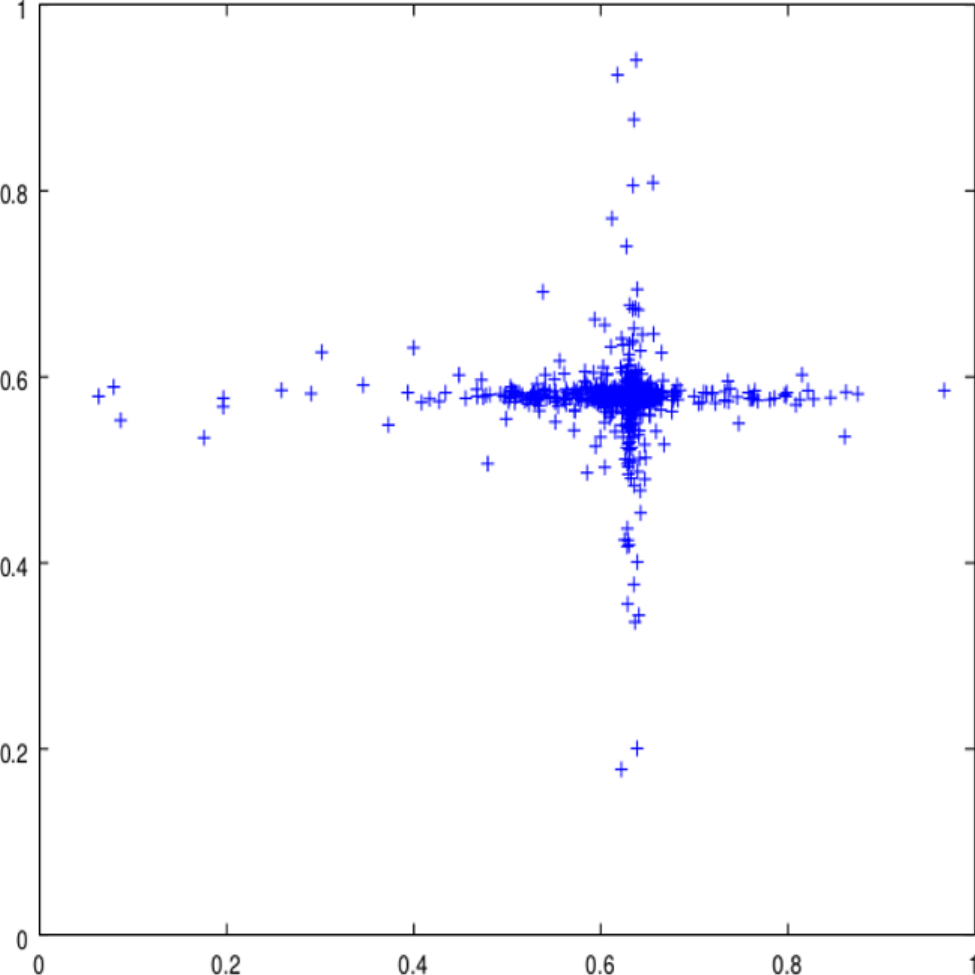} 
 \hspace{0.2cm} \includegraphics[height=3.2cm]{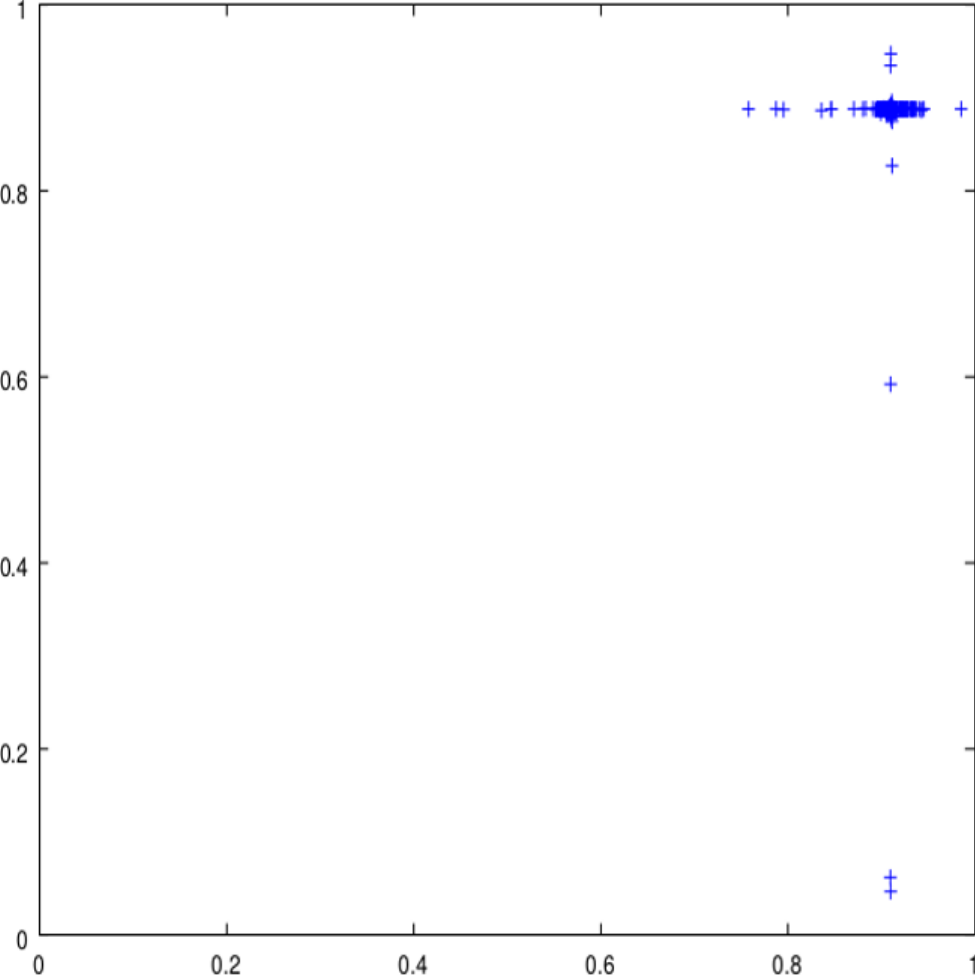} \hspace{0.2cm} \includegraphics[height=3.2cm]{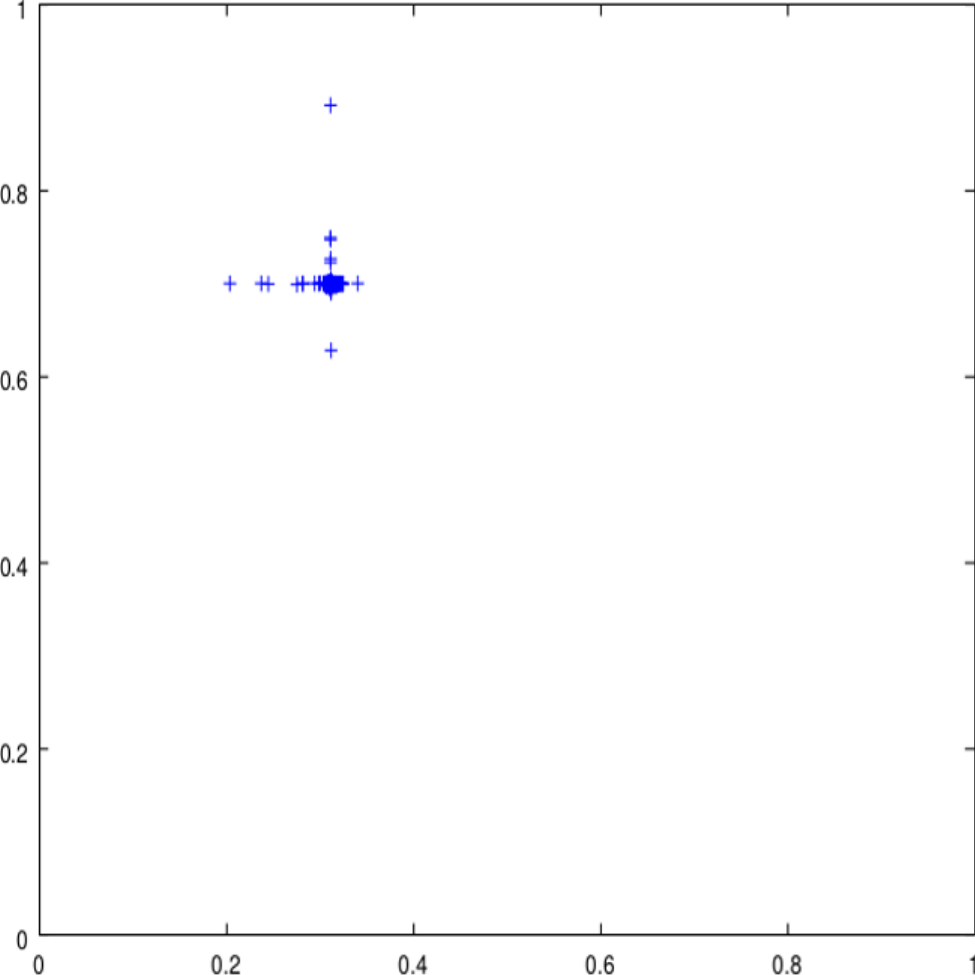} 
\caption{ Iterates of 1000 randomly chosen initial points under random compositions of the  
torus diffeomorphisms $f_0(x,y) = (x + \frac{1}{4\pi} \sin (2\pi x) , y + \frac{1}{4\pi} \sin (2\pi y))$
and $f_1 (x,y) = (x + \frac{1}{\pi} , y + \frac{1}{\sqrt{2}\pi}) \mod 1$,
picked randomly with probability $1/2$ for both. The randomly chosen initial points in the left frame
are iterated $50$, $100$, $150$ iterates in the following frames from left to right.
\label{f:synchonT2}}
\end{center}

\end{figure}
Synchronization phenomena as illustrated in Figures~\ref{f:synchonT} and \ref{f:synchonT2} and as considered in this paper fall into
the larger concept of master-slave synchronization or synchronization by noise.
The classical concept of synchronization is the phenomenon that different oscillations in coupled systems  
will converge to oscillations that move with identical frequency.
It has been realized that external forcing or noise, rather than coupling, can also synchronize dynamics. 
An example of synchronization by noise that applies to iterated function systems  
on the circle is by Antonov  \cite{ant84}.
An illustrative example for linear differential equations forced by the Lorenz equations,  
is given by Pecora and Carroll \cite{pc90}.  
Other examples are from random interval diffeomorphisms \citep{ghahom16},
random logistic maps \cite{ste01}
and stochastic differential equations \citep{masyou02,flagessch17a,flagessch17b}.
The book \cite{synch} 
contains an excellent overview of different aspects of synchronization and includes discussions of synchronization by external forces or noise.

In the context of iterated maps, master-slave synchronization involves dynamics 
\begin{align}\label{e:f}
x(k+1) &= f(y(k),x(k))
\end{align} 
for a state variable 
$x(k) \in M$ and a driving system 
\begin{align}
\label{e:g} 
y(k+1) = g(y(k))
\end{align} 
on a space $N$.
The entire dynamics $(y(k+1),x(k+1))  = F(y(k),x(k))$ with
\begin{align*}
F (y,x) &= (g(y),f(y,x)) 
\end{align*} thus is
a skew product system on $N\times M$ with base space $N$ and fibers $\{y\} \times M$.
Master-slave synchronization is the effect that typical orbits  of \eqref{e:f}
converge to each other under the same driving dynamics, i.e.
identical orbits of \eqref{e:g}:
\begin{align*}
  \lim_{n\to\infty} d (\Pi F^n (y,x_1) , \Pi F^n(y,x_2)) =0,
\end{align*}
where $\Pi(y,x)=x$ and 
 $d$ denotes the distance on $M$.
The effect can be explained by the existence of an attracting invariant graph
for the skew product system \citep{MR1605989,sta99}. 
Iterated function systems fit this description: they are studied
using a formulation as a skew product system over a shift operator (see below).

\subsection{Minimal iterated function systems}

One ingredient of this study is the robust occurrence  of minimal iterated function systems.
Consider a  collection $\mathcal{F} = \{ f_0,\ldots,f_{k} \}$ of diffeomorphisms on a compact manifold $M$.

\begin{definition}
The iterated function system $\text{IFS}\,  (\mathcal{F})$ generated by $\mathcal{F}$ is the 
set of all possible finite compositions $f_{\omega_n} \circ \cdots \circ f_{\omega_0}$ of maps in $\mathcal{F}$, that is, the semi-group generated by these maps.


A $C^1$ neighborhood of $\text{IFS}\,  (\mathcal{F})$  consists of iterated function systems $\text{IFS}\, (\mathcal{G})$, 
where 
$\mathcal{G} = \{ g_0,\ldots,g_k\}$ with $g_i \in U_i$ and $U_i$, $0\le i \le k$, 
is a collection of $C^1$ open neighborhoods $U_i$ of $f_i$.
\end{definition}

%


\begin{definition}
An iterated function system $\text{IFS}\, (\mathcal{F})$, $\mathcal{F} = \{ f_0,\ldots,f_k\}$, on $M$ is called minimal if for every $x \in M$ there exists a sequence of compositions
$f_{\omega_n} \circ \cdots \circ f_{\omega_0}$  so that the sequence 
$f_{\omega_n} \circ \cdots \circ f_{\omega_0} (x)$ is dense in $M$. 

The iterated function system $\text{IFS}\, (\mathcal{F})$ is $C^1$-robustly minimal if there exists a $C^1$ neighborhood $U$
of it,
%
%
%
so that 
each $\text{IFS}\, (\{g_0,\ldots,g_{k}\} )$ from $U$ is minimal.
\end{definition}

By \cite{homnas}, 
any compact manifold admits a pair of diffeomorphisms $f_0,f_1$ that generates a $C^1$-robustly minimal iterated function system.

\subsection{Synchronization in iterated function systems}

Consider finitely many diffeomorphisms $\mathcal{F} = \{ f_0, \ldots, f_{k}  \}$ on a compact manifold $M$ and
fix positive probabilities $p_i$, $0\le i \le k$, with $\sum_{i=0}^k p_i  =1$.
The diffeomorphisms $f_i$ are picked at random, independently at each iterate, with probability $p_i$.
In this context of Markov processes a central notion is that of stationary measure. 

\begin{definition}
A stationary measure $m$ for the iterated function system $\text{IFS}\, (\mathcal{F})$
is a probability measure that is equal to its average pushforward under the diffeomorphisms:
\begin{align}\label{e:defstat}
m &= \sum_{i=0}^{k} p_i (f_i)_* m,
\end{align}
where $(f_i)_* m$ is the pushforward measure $(f_i)_* m (A) = m (f_i^{-1} (A))$. 
\end{definition}

It is well known that $\text{IFS}\, (\mathcal{F})$ always admits at least one stationary measure, 
because $M$ is a compact space.

\begin{lemma}\label{l:fullsupport}
For an iterated function system that is minimal on $M$, a stationary measure $m$ has full support $M$.
\end{lemma}

\begin{proof}
See \cite[Proposition~5]{1086.37026}.
The support $\text{supp}\,m$ of $m$ is closed and nonempty. 
From \eqref{e:defstat} we get \[\text{supp}\,m = \cup_{i=0}^{k} f_i (\text{supp}\,m),\]
so that $\text{supp}\,m = M$ by minimality.
\end{proof}

For the given probabilities $p_i$ on $\{ 0, \ldots, k\}$,  
let $\nu^+$ be the product measure (or Bernoulli measure) on \[\Sigma^+_{k+1} = \{ 0,\ldots,k\}^\mathbb{N}.\] 
For $\omega \in \Sigma_{k+1}^+$ we write
\[f_\omega^n (x) = f_{\omega_n} \circ \cdots \circ  f_{\omega_1} \circ f_{\omega_0} (x).\]
The following result provides conditions for an iterated functions system by diffeomorphisms to display synchronization.
We first recall the notion of fiber Lyapunov exponent.
With the stationary measure $m$ as in the statement of Theorem~\ref{t:synch}, and in light of Lemma~\ref{l:corres} below, 
one has that for $(\nu^+ \times m)$-almost all $(\omega,x) \in \Sigma^+_{k+1} \times M$, and $0 \ne v \in T_x M$,  
\[
 \lim_{n\to \infty}  \frac{1}{n} \ln \|  Df^n_\omega (x) v  \|
\]
exists. 
The number of limit values, counting multiplicity, equals the dimension of $M$.
The possible limit values are the fiber Lyapunov exponents. 
If $\nu^+ \times m$ is ergodic (the stationary measure $m$ is then also called ergodic),
 the fiber Lyapunov exponents are independent of $(\omega,x)$.
We refer to e.g. \cite{via14} for more information.

\begin{theorem}\label{t:synch}
	Let $\text{IFS}\, (\mathcal{F})$, $\mathcal{F} =  \{ f_0,\ldots,f_k\}$ be an iterated function system of $C^2$ diffeomorphisms
	on $M$, where $f_i$ is picked with probability $p_i > 0$.
		There is $k \ge 1$ and 
	a $C^1$ open set of iterated function systems $\text{IFS}\, (\mathcal{F})$, $\mathcal{F}= \{ f_0,\ldots,f_k\}$,
	generated by $C^2$ diffeomorphisms $f_i$, $0\le i\le k$,  on $M$
	and picked with probabilities $p_i >0$,
	with the following properties.
		\begin{enumerate}[label=(\roman*)]
		\item\label{i:i}  $\text{IFS}\, (\mathcal{F})$ is minimal; 
		\item\label{i:ii} 
		$\text{IFS}\, (\mathcal{F})$ admits an ergodic stationary measure $m$ (which is of full support by Lemma~\ref{l:fullsupport});
			\item\label{i:ciii} the diffeomorphism $f_0$ has an attracting fixed point $Q_0$ with $m (W^s (Q_0) ) > 1/2$;
				\item \label{i:iii} $\text{IFS}\, (\mathcal{F})$  has only negative fiber Lyapunov exponents.
		\end{enumerate}
The following properties hold for any iterated function system from this open set.		
	\begin{enumerate}
	\item\label{i:iv}
	For $\nu^+$-almost all $\omega \in \Sigma^+_{k+1}$ there is an open, dense set $W(\omega)\subset M$  so that 
	\begin{align*} 
	\lim_{n\to \infty} d( f^n_{\omega} (x) , f^n_{\omega} (y)) = 0
	\end{align*}
	for  $x,y \in W(\omega)$;
	\item  
	The stationary measure $m$ is the unique stationary measure.
\end{enumerate}
\end{theorem}

To obtain the synchronization property (\ref{i:iv}) we make use of the theory of nonuniform hyperbolicity;
for this reason we require $f_0, \ldots, f_k$ to be $C^2$ diffeomorphisms and not just $C^1$.

For diffeomorphisms on the circle Theorem~\ref{t:synch} is known to hold for iterated function systems with 
two generators, i.e. with $k=1$ (see \cite{1086.37026}). We will show in Proposition~\ref{p:synchM2} that
also on compact surfaces it holds with $k=1$.

The arguments that have been applied to obtain synchronization for iterated function systems on the circle 
used specific properties of one dimensional systems, such as the property that connected sets of 
small measure have small diameter. This is not true in higher dimensions, so that a partly different approach was needed.
We use the theory of nonuniformly hyperbolic systems, in particular the existence of local stable manifolds
in cases of negative fiber Lyapunov exponents. 
There are three main conditions that we use to establish synchronization: 
a minimal iterated function system, negative fiber Lyapunov exponents yielding local contraction, 
and global contraction in the form of the existence of an open set of sufficiently large (stationary) measure
that is mapped into an arbitrary small ball by suitable iterates.

\subsection{Invariant measures for skew product systems}

Compositions of diffeomorphisms $f_0,\ldots,f_k$ on $M$ can be studied in a single framework given by a skew product system
$F^+:  \Sigma_{k+1}^+ \times M \to \Sigma_{k+1}^+ \times M$, 
\begin{align*}
 F^+ (  (\omega_i)_0^\infty , x) = ((\omega_{i+1})_0^\infty  , f_{\omega_0} (x)). 
\end{align*}
Indeed, the coordinate in $M$ iterates as
\[
x , f_{\omega_0} (x) , f_{\omega_1} \circ f_{\omega_0} (x) , f_{\omega_2} \circ f_{\omega_1} \circ f_{\omega_0} (x), \ldots
\]
We will write
\begin{align*}
F^+ (\omega , x ) &= (\sigma \omega , f_{\omega} (x) )
\end{align*}
with the shift operator $(\sigma \omega)_i = \omega_{i+1}$ on $\Sigma^+_{k+1}$.
As the dependence of the fiber maps $f_\omega$ on $\omega$ is on $\omega_0$ alone, 
the skew product maps are of a restricted kind called step skew product maps.

An explicit computation on sets that generate the Borel sigma-algebra shows the following connection, which is standard 
(see e.g. \cite{ghahom16}).

\begin{lemma}\label{l:corres}
A probability measure $m$ is a stationary measure if and only if
  $\mu^+ = \nu^+ \times m$
is an invariant measure of $F^+$ with marginal $\nu^+$ on $\Sigma^+_{k+1}$. 
\end{lemma}

We call $m$ an ergodic stationary measure if $\nu^+ \times m$ is ergodic for $F^+$.
The natural extension of $F^+$ is obtained when the shift acts on two sided time $\mathbb{Z}$;
this yields a skew product system $F : \Sigma_{k+1} \times M\to \Sigma_{k+1}\times M$
with 
\[
\Sigma_{k+1} = \{0,\ldots,k\}^\mathbb{Z}
\] 
and given by the same expression
\begin{align*}
F (\omega , x ) &= (\sigma \omega , f_{\omega} (x) ).
\end{align*}
Recall the notation
\[F^n(\omega , x) = (\sigma^n \omega , f^n_\omega(x) ) = 
(\sigma^n \omega , f_{\omega_{n-1}} \circ \cdots \circ f_{\omega_0} (x))\] for iterates of $F$.

Invariant measures for $F^+$ with marginal $\nu^+$ and invariant measures for $F$ with marginal $\nu$
are in one to one relationship.
We quote the following result that precises this correspondence. 
  Write $\mathcal{M}$ for the space of probability measures on $M$ endowed with the weak star topology.
  
 \begin{proposition}\label{p:furstenberg}
 Let $m$ be a stationary measure for $\text{IFS}\, (\mathcal{F})$.
 Then there exists a measurable map $\mathcal{L} : \Sigma_{k+1}  \to \mathcal{M}$,
 such that
 $$
  f^n_{\sigma^{-n}\omega}   m
         \to {\mathcal L}({\omega})
 $$
 as $n\to \infty$, $\nu$-almost surely.
 The measure $\mu$ on $\Sigma_{k+1} \times M$ 
 with marginal $\nu$ and conditional measures
 $\mu_\omega = \mathcal{L}({\omega})$ is an $F$-invariant measure.
 \end{proposition}
 
 \begin{proof}
 See 
 \cite[Theorem 1.7.2]{arn98} for a general result that implies the proposition, 
 or \cite[Appendix~A]{ghahom16} or \cite[Chapter~5]{via14} for this setting.
 \end{proof}

A stationary measure $m$ thus, through the invariant measure $\nu^+ \times m$ for $F^+$, gives rise to an invariant measure
$\mu$ for $F$, with marginal $\nu$.
The measure $\mu$ has conditional measures $\mu_\omega$, meaning
\begin{align*}
\mu (A) &= \int_{\Sigma_{k+1}} \mu_\omega (A \cap (\{\omega\} \times M))\, d \nu (\omega)
\end{align*}
for Borel sets $A$.

\begin{proposition}\label{p:crauel}
Let $m$ be an ergodic stationary measure for $\text{IFS}\, (\mathcal{F})$.
 Assume $F^+$  has negative fiber Lyapunov exponents with respect to the
 ergodic measure $\nu^+ \times m$.
 Then the conditional measures $\mu_\omega$ for the $F$-invariant measure $\mu$ 
 are a finite sum of $K$ delta measures of equal mass $1/K$, for $\nu$-almost all $\omega$.
 \end{proposition}
 
 \begin{proof}
 See \citep{cra90,lej87}. 
 \end{proof}
 

Recent work by Bochi, Bonatti and D\'{\i}az \cite{bbd} establishes for each compact manifold $M$ 
a $C^2$ open set of iterated function systems,
generated by finitely many diffeomorphisms on $M$, for which the corresponding skew product system on $\Sigma_{k+1} \times M$
admits an invariant measure of full support for which all fiber Lyapunov exponents are zero. 
The following result phrases Theorem~\ref{t:synch} in similar terms.

\begin{proposition}\label{p:measure}
Let $M$ be a compact manifold.
There is $k \ge 1$ and a $C^1$ open set of minimal iterated function systems
generated by $C^2$ diffeomorphisms $f_0$, $0 \le i \le k$, on $M$
and picked with positive probabilities $p_i$, 
with the following properties.

The iterated function system admits a unique stationary measure  $m$. 
The corresponding invariant measure $\mu$ for the skew product system $F$ on $\Sigma_{k+1} \times M$
has the following properties:
\begin{enumerate}[label=(\roman*)]
 \item $F$  has only negative fiber Lyapunov exponents with respect to $\mu$;
 \item $\mu$ has full support;
 \item the marginal of $\mu$ on $\Sigma_{k+1}$ is the Bernoulli measure $\nu$;
 \item the conditional measures $\mu_\omega$ are delta measures: $\mu_\omega = \delta_{X(\omega)}$ for a measurable map
 $X: \Sigma_{k+1} \to M$.
 \end{enumerate}
\end{proposition} 
 
The open class of iterated function systems in \cite{bbd} is given in terms of conditions
which they term minimality (of an induced iterated function system on a flag bundle) and maneuverability.
The construction in \cite{bbd} makes clear that these conditions can occur simultaneously with 
the conditions defining the set of iterated function systems in Proposition~\ref{p:measure}.
So Proposition~\ref{p:measure} combined with \cite{bbd} yields the following result.
 
\begin{proposition}\label{p:measure2}
Let $M$ be a compact manifold.
There is $k \ge 1$ and a $C^2$ open set of minimal iterated function systems
generated by diffeomorphisms $f_i$, $0\le i \le k$, on $M$
and picked with positive probabilities $p_i$,
with the following properties.

The corresponding skew product system $F$ on $\Sigma_{k+1} \times M$ admits simultaneously 
\begin{enumerate}[label=(\roman*)]
\item an invariant measure $\mu$ that has full support, Bernoulli measure as marginal, delta measures as conditional measures on fibers, and negative fiber Lyapunov exponents;
\item an invariant measure $\nu$ that has full support and zero fiber Lyapunov exponents.
\end{enumerate}
\end{proposition} 


The techniques in this paper may possibly be extended from the step skew product systems
given by iterated function systems to more general skew product systems. An example is in \cite{hom}, 
where the ideas of synchronization have been used to clarify properties of the disintegrations of Lebesgue measure along
center manifolds in certain conservative partially hyperbolic dynamical systems.\\

\noindent {\bf Acknowledgments.}
I am grateful to Masoumeh Gharaei for many discussions on the paper.

\section{Proofs}

The proofs of Theorem~\ref{t:synch} and of Proposition~\ref{p:measure} contain  different steps presented as lemmas that are grouped
in sections. The sections 2.i below treat the following steps:
\begin{description}
 \item[{\rm 2.1}] the existence of minimal iterated function systems with negative fiber Lyapunov exponents 
 (proving items~\ref{i:i}, \ref{i:iii} in Theorem~\ref{t:synch});
  \item[{\rm 2.2}] properties of the stationary measure that are needed in the proof;
 \item[{\rm 2.3}] the existence of atomic conditional measures of the $F$-invariant measure $\mu$ along fibers 
 (finishing the proof of Proposition~\ref{p:measure});
 \item[{\rm 2.4}] the occurrence of synchronization 
 (item~\ref{i:iv} in Theorem~\ref{t:synch});
 \item[{\rm 2.5}] the uniqueness of the stationary measure (item~\ref{i:ii} in Theorem~\ref{t:synch}, finishing the proof of Theorem~\ref{t:synch}).
 \end{description}

\subsection{Minimality and negative fiber Lyapunov exponents}\label{ss:miniplusneg}

This section constructs a $C^1$ open set of  minimal iterated function systems with negative fiber Lyapunov exponents.
We start with a number of results on iterated function systems generated by two maps.
We collect statements from \cite{homnas} 
that we need in the sequel.

Recall that a Morse-Smale diffeomorphism on $M$ is a diffeomorphism whose recurrent set consists of finitely many
fixed and periodic points, all hyperbolic. Moreover, their stable and unstable manifolds are transverse.

\begin{lemma}\label{l:homnasuitgewerkt}
Let $M$ be a compact manifold. Then there is a pair of diffeomorphisms $\hat{g}_0,\hat{g}_1$ on $M$ that generates a $C^1$-robustly minimal iterated function system. The diffeomorphisms satisfy the following properties.
The diffeomorphism $\hat{g}_0$ is a Morse-Smale diffeomorphism with a unique attracting fixed point $Q_0$, 
whose basin of attraction is open and dense in $M$.
There is a small neighborhood $U_0$ of $Q_0$ and a compact ball $B_0\subset U_0$ with the following properties:
\begin{enumerate}[label=(\roman*)]
\item $\hat{g}_1$ has a repelling fixed point in $B_0$;
\item  $\hat{g}_0$ and $\hat{g}_0 \circ \hat{g}_1$ are contractions on $U_0$, mapping $U_0$ into $U_0$;
\item $B_0 \subset \hat{g}_0 (B_0) \cup \hat{g}_0\circ \hat{g}_1 (B_0)$;
\end{enumerate}
\end{lemma}

Denote $h_0 = \hat{g}_0, h_1 = \hat{g}_0\circ \hat{g}_1$.
By classical theory of iterated function systems \cite{hut81},
there is a unique compact set $S$ with $B_0 \subset S \subset U_0$ and 
\begin{align}\label{e:Sinv} 
S &= h_0 (S) \cup h_1 (S).
\end{align}
Further,
\begin{align}\label{e:diam0}
 \text{diam}\, h^n_\omega (S) &\to 0
\end{align} 
for all $\omega \in \Sigma_2^+$ and 
uniformly in $n$.
By \eqref{e:Sinv}, for each $n \in \mathbb{N}$, 
\begin{align*}
S &= \cup_{\omega \in \Sigma_2^+} h^n_\omega (S).
\end{align*} 
With \eqref{e:diam0} this implies that $\text{IFS}\, (\{\hat{g}_0,\hat{g}_0\circ \hat{g}_1\})$ is minimal on $S$.

To obtain minimal iterated function systems with negative fiber Lyapunov exponents, we start
with an iterated function system as in Lemma~\ref{l:homnasuitgewerkt} and modify to bring
strong contraction on a region. This is elaborated in the proof of Lemma~\ref{l:existence} below.
In the proof we need a statement on the dependence of stationary measures on the iterated function system, which we provide first.
Recall that $\mathcal{M}$ denotes the space of probability measures on $M$ endowed with the weak star topology.
Take a metric $d_\mathcal{M}$ that generates the weak star topology, see e.g.
\cite{MR889254}.
Denote by $\mathcal{P} : \mathcal{M} \to \mathcal{M}$ the map
\begin{align*}
 \mathcal{P} m &= \sum_{i=0}^{k} p_i (f_i)_* m.
\end{align*}
Note that stationary measures  are fixed points of $\mathcal{P}$.

\begin{lemma}\label{l:c0}
The map $\mathcal{P}$ is continuous.
It also depends continuously on $f_0,\ldots,f_{k}$ if these vary in the $C^0$ topology.
\end{lemma}

For the proof we refer to \cite{ghahom16}.

\begin{lemma}\label{l:existence}
There exists a  $C^1$ open set of iterated function systems $\text{IFS}\, (\{ g_0,g_1\})$ 
so that $\text{IFS}\, (\{ g_0,g_1\})$ in minimal on $M$ and, for each stationary measure $m$,
has negative fiber Lyapunov exponents. 
\end{lemma}

\begin{proof}
Start with 
$\hat{g}_0,\hat{g}_1$ satisfying the properties listed in Lemma~\ref{l:homnasuitgewerkt}. 
Let a smooth map $\tilde{g}_0$ and open balls $D_0 \subset C_0 \subset B_0\subset U_0$ be so that
\begin{enumerate}[label=(\roman*)]
\item $\tilde{g}_0 = \hat{g}_0$ on $B_0 \setminus C_0$;
\item  $D\tilde{g}_0 = 0$ on $D_0$;
\item $\tilde{g}_0$ and $\tilde{g}_0 \circ \tilde{g}_1$ are contractions on $U_0$, mapping $U_0$ into $U_0$;
\item $B_0 \subset \tilde{g}_0 (B_0) \cup \tilde{g}_0\circ \tilde{g}_1 (B_0)$;
\end{enumerate}
Because of the vanishing derivative on $D_0$, $\tilde{g}_0$ is not a diffeomorphism.
Such maps can be constructed by modifying the constructions in \cite{homnas} as follows.
By working in a chart containing $U_0$ we may assume $\hat{g}_0,\hat{g}_1$ are diffeomorphisms on Euclidean space $\mathbb{R}^n$,
with $B_0$ containing the origin.
Let $\phi: \mathbb{R} \to \mathbb{R}$ be a smooth odd function with $\phi \equiv 0$ on a small neighborhood of $0$, and otherwise increasing  plus equal to the identity outside a neighborhood of $0$. For $s>0$ small, let $r : \mathbb{R}^n \to \mathbb{R}^n$ be given by $ r(x) = s \phi (\| x/s\|)$
and let $\tilde{g}_0 = \hat{g}_0 \circ r$.

Although  $\tilde{g}_0$ is not a diffeomorphism,
there are 
diffeomorphisms in any $C^1$-neighborhood of it:
to perturb to a diffeomorphism it suffices to perturb $\phi$ to a smooth increasing function.
Keep $\hat{g}_1$ unchanged and write $\tilde{g}_1 = \hat{g}_1$.

%

For $s$ small, the properties listed in Lemma~\ref{l:homnasuitgewerkt} remain true for $\text{IFS}\, (\{\tilde{g}_0,\tilde{g}_1\})$. 
By identical arguments: there is a unique compact set $\tilde{S}$ with $B_0\subset \tilde{S} \subset U_0$ and
\[
\tilde{S} = \tilde{g}_0 (\tilde{S}) \cup \tilde{g}_0 \circ \tilde{g}_1 (\tilde{S}).
\]
Moreover,  $\text{IFS}\, (\{ \tilde{g}_0,\tilde{g}_0\circ \tilde{g}_1 \})$ is minimal on $\tilde{S}$.
Since $\tilde{g}_0, \tilde{g}_1$ are equal to $\hat{g}_0,\hat{g}_1$ outside $U_0$,
$\text{IFS}\,( \{\tilde{g}_0,\tilde{g}_1\})$ is minimal on $M$, compare \cite{homnas}.

By Lemma~\ref{l:c0}, the set of fixed points  of $\mathcal{P}$ is a compact subset of $\mathcal{M}$  that varies 
upper semi-continuously when varying $\tilde{g}_0,\tilde{g}_1$ in the $C^1$ topology.
That is, for any neighborhood $O\subset \mathcal{M}$ of the closed set of stationary measures of  $\text{IFS}\,(\{\tilde{g}_0,\tilde{g}_1\})$,
there is a $C^1$ neighborhood of $\text{IFS}\,(\{\tilde{g}_0,\tilde{g}_1\})$ so that each iterated function system 
from it has its stationary measures contained in $O$.

We claim the existence of an open neighborhood of $\tilde{g}_0,\tilde{g}_1$ in the $C^1$ topology, so that for all pairs of 
diffeomorphisms $g_0,g_1$ from it,  and for any ergodic stationary measure $m$ of  $\text{IFS}\, \{ g_0,g_1\}$,
$\text{IFS}\, \{ g_0,g_1\}$  has negative fiber Lyapunov exponents.

Suppose otherwise. Then there is a sequence $\tilde{g}_{0,j}$ converging to $\tilde{g}_0$ 
and  $\tilde{g}_{1,j}$ converging to $\tilde{g}_1$
in the $C^1$ topology, with a nonnegative fiber  Lyapunov exponent for some ergodic stationary measure $\tilde{m}_j$. 
By passing to a subsequence we may assume that $\tilde{m}_j$ converges to a measure $\tilde{m}$.
Lemma~\ref{l:c0} shows that $\tilde{m}$ is a stationary measure of $\text{IFS}\,(\{ \tilde{g}_0,\tilde{g}_1 \} )$.
By Lemma~\ref{l:fullsupport}, $\tilde{m}$ has full support. (For these statements it plays no role that $\tilde{g}_0$ is not a diffeomorphism.)

The top fiber  Lyapunov exponent $\lambda_{1,j}$ of $\text{IFS}\,(\{ \tilde{g}_{0,j},\tilde{g}_{1,j}\})$
satisfies, for $(\nu^+\times \tilde{m}_j)$-almost all $(\omega,x)$,
 \begin{align*} \lambda_{1,j} &\le 
  \lim_{n\to \infty} \frac{1}{n}  \sum_{i=0}^{n-1}  \ln \|  D\tilde{g}_{ \sigma^i \omega,j} ( \tilde{g}_{\omega,j}^i (x) )   \| \\ &=  
  \int_M \int_{\Sigma^+_{2}}   \ln \|  D\tilde{g}_{\omega,j} \|\, d \nu^+ (\omega)\, d \tilde{m}_j
  \\ &=   \int_M p_0 \ln \|  D\tilde{g}_{0,j} \|  + p_1 \ln \|  D\tilde{g}_{1,j} \| \, d \tilde{m}_j.
  \end{align*} 
  We can bound
  $\ln \|D\tilde{g}_{0,j}\| \le C$ on $M$ uniformly in $j$.
  There is also a bound
  $\ln \| D \tilde{g}_{0,j} \| \le B_j $ on $C_0$ with $B_j \to -\infty$ as $j\to \infty$.
So we get   
\begin{align}\label{e:CBj}
\int_{M}  \ln \| D\tilde{g}_{0,j} \|\, d \tilde{m}_j &\le C +  B_j \tilde{m}_j(C_0).
\end{align}
As $\tilde{m}_j \to \tilde{m}$ for $j \to \infty$, 
\begin{align*}
\liminf_{j \to \infty} \tilde{m}_j (C_0) &\ge \tilde{m} (C_0) > 0,
\end{align*}
see e.g. \cite[Theorem~III.1.1]{shi84}.
As further  $B_j \to -\infty$ as $j \to \infty$, we conclude from \eqref{e:CBj} that 
\[\int_{C_0}   \ln \| D\tilde{g}_{0,j} \|\, d \tilde{m}_j \to -\infty\] as $j \to \infty$.
Since also  $\int_M \ln \|  D\tilde{g}_{1,j} \|\, d \tilde{m}_j$ is bounded uniformly in $j$,
it  follows that
\[ \lim_{j\to \infty} \lambda_{1,j} = - \infty.\]
This contradiction proves the lemma.
\end{proof}

\subsection{Regularity of stationary measures}

Let $g_0,g_1$ be Morse-Smale diffeomorphisms as in Lemma~\ref{l:existence}.
Recall that the attracting fixed point $Q_0$ of $g_0$ has a basin of attraction 
$W^s(Q_0)$ that is open and dense in $M$.


The reasoning in the following sections would work for $\text{IFS}\, (\{f_0,f_1\})$ with $f_0 = g_0$ and $f_1=g_1$ (i.e. $k=1$ in Theorem~\ref{t:synch} and Proposition~\ref{p:measure}) 
if the basin $W^s (Q_0)$ has sufficiently large stationary measure;
$m (W^s (Q_0) ) > 1/2$. As we do not know whether this is the case, we will add sufficiently many  diffeomorphisms, all  
small perturbations of $g_0$, as generators of an iterated function system and we bound the stationary measures of the
basins of the  attracting points of these extra generators.  It turns out that at least one of these basins 
has stationary measure more than $1/2$, which will suffice for the reasoning in the following sections.   

The complement $\Lambda_0 = M \setminus W^s (Q_0)$ 
is a stratification consisting of the stable manifolds of finitely many hyperbolic fixed or periodic points.

\begin{definition}
A stratification is a compact set consisting of finitely many manifolds $W_i$
with 
\begin{enumerate}[label=(\roman*)]
\item $W_0$ is closed;
\item $\mathrm{dim}\; W_{i+1} \ge \mathrm{dim}\; W_i$;
\item $\overline{W_{i+1}} \backslash W_{i+1} \subset W_0 \cup \ldots \cup W_i$;
\item if $x_n \in W_j$ converges to $y \in W_i$, then there is a sequence of $d$-planes 
$E_n \subset T_{x_n} W_j$
of dimension $d = \mathrm{dim}\; W_i$ that converge to $T_y W_i$. 
\end{enumerate}
\end{definition}

\begin{definition}
Two stratifications $N_1,N_2$ inside $M$ are transverse if at intersection points
the tangent spaces span the tangent space of $M$.

A collection of stratifications $N_1,\ldots,N_l$  is said to be transverse at a common intersection point  
if any $N_k$ is transverse to any intersection $\cap_j  N_{i_j}$ of a subcollection not containing $N_k$.
A collection of stratifications $N_1,\ldots,N_l$ is transverse if any subcollection is transverse at a
common intersection points of the subcollection.
\end{definition}

Starting point for the following is a robustly minimal iterated function system 
$\text{IFS}\, (\{ g_0,\ldots, g_0, g_1\})$, with $k$ copies of $g_0$. Given are probabilities $p_0, \ldots, p_k$
to pick the diffeomorphisms from.  We will assume that   $\text{IFS}\, (\{ g_0,\ldots, g_0, g_1\})$
has robustly negative fiber Lyapunov exponents. 
It follows from the discussion in Section~\ref{ss:miniplusneg} that for given positive probabilities $p_i$, such an iterated function system exists.

\begin{lemma}\label{l:genplusstrat}
Let $k$ be a positive integer.
In any neighborhood of $\text{IFS}\, (\{ g_0,\ldots, g_0, g_1\})$ with $k$ copies of $g_0$, there  
is an open set of iterated function systems
 so that for each
$\text{IFS}\, (\mathcal{F})$, $\mathcal {F} = \{ f_0,\ldots,f_{k} \}$ from this open set,
\begin{enumerate}[label=(\roman*)]
\item $\text{IFS}\, (\mathcal{F})$ is minimal;
\item $\text{IFS}\, (\mathcal{F})$  has negative fiber Lyapunov exponents for each stationary measure;
\item \label{i:qi} the diffeomorphism $f_i$, $0 \le i \le k-1$, has a unique attracting fixed point $Q_i$ 
with open and dense basin $W^s (Q_i)$;
\item \label{i:li} with $\Lambda_i = M \setminus W^s (Q_i)$, the collection
$\{ \Lambda_i\}$, $0\le i \le k-1$,  is a transverse collection of stratifications. 
\end{enumerate}
\end{lemma}

\begin{proof}
%
We already have the first two items. We must check the remaining items~\ref{i:qi} and \ref{i:li}.
Item~\ref{i:qi} is  fulfilled for any $f_i$ sufficiently close to $g_0$ since $g_0$ is a Morse-Smale diffeomorphism.
So it remains to find an open set of diffeomorphisms for which item~\ref{i:li} holds.
We refer to \cite{Hi}, see in particular \cite[Exercise~3.15]{Hi},
for the transversality theorem for stratifications.
It implies that for  a $C^1$ open and dense set of $k$ diffeomorphisms $f_i$, $0\le i \le k-1$, 
the collection of stratifications $\Lambda_i$ is transverse.
\end{proof}

The following lemma bounds the stationary measure on the stratifications. 

\begin{lemma}\label{l:m=0:gen}
Let $f_0,\ldots,f_k$ be diffeomorphisms as in Lemma~\ref{l:genplusstrat}, so that the collection of stratifications $\Lambda_i$ is transverse.
For $k> 2 \dim(M)$, any stationary measure $m$ satisfies
$m(\Lambda_i) < 1/2$ for some $0\le i < k$. 
\end{lemma}

\begin{proof}
We will show that if $k$ is large enough, any probability measure on $M$
satisfies  $m(\Lambda_i) < 1/2$ for some $0\le i < k$. 
Assume $m$ is a probability measure on $M$ with
\begin{align*} 
m(\Lambda_i) &\ge 1/2 
\end{align*}
 for all $i$.
The smallest possible total measure on a union of stratifications
$\Lambda_{i_1} \cup \ldots \cup \Lambda_{i_l}$, varying over the probability measures on $M$, occurs if the measure is supported on the common intersection,
if this is nonempty. By transversality we have that for $l = \dim(M)$ this intersection, if nonempty, is zero-dimensional.
Also, the intersection of $\dim(M)+1$ different stratifications is always empty.

To calculate the smallest possible total measure on $\Lambda_0\cup\ldots\cup \Lambda_{k-1}$, 
suppose there is measure $1/2$ on each $\Lambda_i$. 
Consider sets  $S_j$ that occur as maximal intersections of sets $\Lambda_{i_1} \cap \ldots \cap \Lambda_{i_l}$;
meaning such that each intersection with a further stratification $\Lambda_i$ is empty.

Think of an assignment of mass $n_j = m(S_j)$ to the $S_j$'s. 
We seek the minimal total measure, among variation of such assignments. 
The argument will be combinatorial.
For the purpose of bounding the minimal total measure, 
we may assume that each collection of $l$ stratifications, $l = \dim(M)$, has a nonempty intersection 
by possibly adding imaginary intersections.
This indeed only adds possible assignments of mass (previous assignments assign zero measure to the new imaginary intersections), 
hence does not increase total measure.

The minimal possible total measure $\sum_j n_j$ on $\Lambda_0 \cup \ldots \cup \Lambda_{k-1}$
occurs at an equidistribution among the different disjoint sets $S_j$. 
At equidistribution each $S_{j}$ carries the same measure, say $n_j = n$.
To see that this gives the minimal possible total measure, 
first observe that a convex combination of assignments of mass, preserving the total measure,
is again an assignment of mass. 
Then by symmetry, permuting the sets $S_j$ and averaging assignments, one obtains the equidistribution.
This therefore has minimal total measure.

There are $\binom{k}{l}$ intersections $S_j$ and $\binom{k-1}{l-1}$ intersections in a fixed stratification $\Lambda_i$.  
At equidistribution, the measure of $S_{j}$ equals  $\frac{1}{2}  /  \binom{k-1}{l-1}$ as the measure of
$\Lambda_{i}$ is  $\frac{1}{2}$. 
The total measure $\sum_j n_j$ is  $ \frac{1}{2} \binom{k}{l} / \binom{k-1}{l-1} = \frac{k}{2l}$. 
This number is bigger than $1$ if $k>2l$.
The lemma follows. 
\end{proof}

We have now constructed iterated function systems $\text{IFS}\, (\mathcal{F})$, $\mathcal {F} = \{ f_0,\ldots,f_{k} \}$,
that are $C^1$-robustly minimal, so that moreover, for any ergodic stationary measure $m$,
the fiber Lyapunov exponents are all negative and (after relabeling the diffeomorphisms if necessary) $f_0$ admits a unique 
attracting fixed point $Q_0$ with $m (W^s(Q_0)) > 1/2$.

\subsection{Delta conditional measures}

Given an ergodic stationary measure $m$, let $\mu$ be the associated invariant measure for $F$ given by Proposition~\ref{p:furstenberg}.
We will establish that the corresponding conditional measures $\mu_\omega$ are delta measures, i.e. $K=1$ in 
Proposition~\ref{p:crauel}.
For this it suffices to establish that $\mu_\omega$ contains a point measure of mass larger than $1/2$, for $\nu$-almost all $\omega$.
Indeed, for each $0 \le c \le 1$,  the set of points $(\omega,x)$ for which $\mu_\omega  = c$ is an invariant set.
By ergodicity this set has $\mu$-measure  equal to $0$ or $1$.
This observation implies the following lemma.


\begin{lemma}\label{l:1/2=1}
Suppose $\mu$ is an ergodic measure for which $\mu_\omega$ contains a point measure of mass larger than $1/2$,
for $\nu$-almost all $\omega$.
Then there is a measurable function $X: \Sigma_{k+1} \to M$ so that
\[
\mu_\omega = \delta_{X(\omega)}.
\] 
\end{lemma}


 %
 %
 Let $d$ denote a distance function from a Riemannian structure on $M$. 
 Write $\Sigma = \Sigma^- \times \Sigma^+$ and $\omega = (\omega^-,\omega^+)$ for 
 $\omega \in  \Sigma^- \times \Sigma^+$.
 The Bernoulli measure $\nu$ on $\Sigma$ can also be written $\nu = \nu^- \times \nu^+$ on
 $\Sigma^- \times \Sigma^+$.
 
 \begin{lemma}\label{l:pesinR}
 For any $\varepsilon >0$, there are $\delta > 0$, $C>0$, $0<\lambda<1$ and a set 
 \[\mathcal{A}   \subset \Sigma_{k+1}\] 
 with $\nu (\mathcal{A})  > 1-\varepsilon$,
 so that for $\omega \in \mathcal{A}$, 
 $\{\omega\} \times M$ contains  a ball $B^s (\omega)$
 of radius $\delta$ 
with   
 \begin{align}\label{e:es}
 d \left( f^n_\omega (x_1),f^n_\omega (x_2) \right) &\le C \lambda^n d(x_1,x_2)
 \end{align}
 for all $n \in \mathbb{N}$, 
 whenever $x_1,x_2 \in B^s (\omega)$.
 Moreover, 
  \[\mathcal{A}  = \Sigma^- \times \mathcal{A}^+ \]
 for a set $\mathcal{A}^+ \subset \Sigma_{k+1}^+$.
 \end{lemma}
 
 \begin{proof}
 The existence of a set $\mathcal{A}\subset \Sigma_{k+1}$ so that \eqref{e:es} holds, follows from 
 the theory of nonuniform hyperbolicity 
 \cite[Lemma~10.5]{MR2068774}.
 
Write $\pi: \Sigma_{k+1} \to \Sigma_{k+1}^+$ for the natural projection.
 Note that the fiber coordinates of $F (\omega,x)$ 
 do not depend on $\omega^-$.
Hence,  
if $\mathcal{A}^+ = \pi \mathcal{A}$,  
when replacing $\mathcal{A}$ by $ \pi^{-1}  \mathcal{A}^+$, estimate \eqref{e:es} still applies.
 That is,  we may consider $\mathcal{A}$ as a product set $\Sigma^- \times \mathcal{A}^+$.
 \end{proof}
 
By Lemma~\ref{l:m=0:gen},
it is possible to take a closed subset 
$D \subset W^s(Q_0)$  so that
\begin{align*}
m(D) &> 1/2
\end{align*}
(after relabeling the diffeomorphisms if necessary).
 Recall that $d_\mathcal{M}$ is a metric on $\mathcal{M}$ that generates the weak star topology.
Let $\Delta \subset \mathcal{M}$ be the subset of probability measures on $M$ that 
assign at least mass $m (D)$ to some point, 
\begin{align*}
\Delta = \{ m \in \mathcal{M} \; \vert \; m (x) \ge m(D) \text{ for some } x\in M\}.
\end{align*}
Note that $\Delta$ is a closed subset of $\mathcal{M}$.
Let $\mathcal{A}$ be a subset of $\Sigma_{k+1}$ as provided by Lemma~\ref{l:pesinR}.


\begin{lemma}\label{l:L}
There exists $L\in \mathbb{N}$ so 
that for each $\omega \in \mathcal{A}$,
there exists $\mathcal{B}_{\omega^+} \subset \Sigma^-$ 
so that for $\zeta \in  \mathcal{B}_{\omega^+} \times \{\omega^+\}$,
\begin{align*}
f^L_{\sigma^{-L} \zeta}  (D) &\subset W^s(B).
\end{align*}
\end{lemma}

\begin{proof}
For any $r>0$, a sufficiently large iterate of $f_{0}$ maps $D$ into a neighborhood of radius $r$
of the attracting fixed point $Q_0$ of $f_{0}$. 
By minimality of $\text{IFS}\, \{ f_0,\ldots, f_k\}$, the set $\cup_{\omega\in D^\mathbb{N}} f^n_\omega (Q_0)$ intersects each open set.
Hence there is, for any $e>0$, an integer $L_1$ so that for any ball $B\subset M$ of diameter
$e$, there are symbols $a_1,\ldots,a_{L_1}$ with $f_{a_{L_1}} \circ \cdots \circ f_{a_1} (Q_0) \in B$.
Combining the above statements,  there is a composition $f_{a_{L_1}} \circ \cdots \circ f_{a_1} \circ f_{0}^{L_2}$ that maps $D$ into 
$B^s (\omega)$. 
We let $\mathcal{B}_{\omega^+}$ consist of the sequences in $\Sigma^-$ that end with these symbols.
Given $L_1$ we may choose $L_2$ so that $L = L_1 + L_2$ does not depend on $\omega$. 
This proves the lemma with $L = L_1 + L_2$.
\end{proof}

The uniform bound on the number of iterates $L$ in the above claim
implies that $\nu^- (\mathcal{B}_{\omega^+})$ is uniformly bounded away from zero.
Consequently the union 
\[\mathcal{B} = \cup_{\omega^+ \in \mathcal{A}^+} \mathcal{B}_{\omega^+} \times \{ \omega^+ \}\] 
has positive measure: \[\nu(\mathcal{B}) >0.\]
By ergodicity of $\nu$, for $\nu$-almost all $\omega$, its orbit under $\sigma^{-1}$ intersects $\mathcal{B}$.
For such $\omega$, Lemma~\ref{l:pesinR} and Lemma~\ref{l:L} yield
\[
\liminf_{n\to \infty} d_\mathcal{M}  ( (f^n_{\sigma^{-n} \omega})_* m , \Delta  ) = 0.
\]
By Proposition~\ref{p:furstenberg} and Lemma~\ref{l:1/2=1}, there is a measurable function $X: \Sigma_{k+1} \to M$ with
\begin{align*}
\lim_{n\to \infty} (f^n_{\sigma^{-n} \omega})_* m = \delta_{X(\omega)}.
\end{align*}
This concludes the proof of Proposition~\ref{p:measure}.


\subsection{Synchronization}

We continue with the statement of Theorem~\ref{t:synch} that describes synchronization (item~\ref{i:iv} in its statement).
For $\nu$-almost all $\omega \in \Sigma_{k+1}$, the fiber Lyapunov exponents at $(\omega, X(\omega))$ exist and are strictly negative. 
 Write $W^s(X(\omega))$ for the stable set of $X(\omega)$ inside the fiber 
 $\{ \omega \} \times M$;
 \[
 W^s( X (\omega)) =  \{  y \in M \; \vert \;  \lim_{n\to \infty} d ( f^n_\omega (y) , X(\sigma^n \omega) ) = 0 \}.
 \]
 The theory of nonuniform hyperbolicity, as in Lemma~\ref{l:pesinR}, yields the following.
 Write $D_\delta (X(\omega))$ for the $\delta$-ball around $X(\omega)$. 
 Then for all $\varepsilon>0$ there is $\delta>0$ so that
 \begin{align*} 
 S(\delta) &= \{ \omega \in \Sigma \; \vert \; D_\delta (X(\omega)) \subset W^s(X(\omega)) \}
 \end{align*}
 satisfies 
 \begin{align}\label{e:S} 
 \nu (S(\delta)) &> 1-\varepsilon.
 \end{align}
 Once orbits are in a $\delta$-ball $D_\delta (X(\omega))$ for $\omega \in S(\delta)$, 
 distances to the orbit of $X(\omega)$ decrease to zero, which we may assume to happen at a uniform rate as in \eqref{e:es}.


\begin{lemma}\label{l:opendense}
For $\nu$-almost all $\omega \in \Sigma_{k+1}$, $W^s(X(\omega))$ is an open and dense subset of $M$.
%
\end{lemma}

\begin{proof}
For $\nu$-almost all $\omega \in \Sigma_{k+1}$,
$W^s (X(\omega))$ is open. Indeed, take $y \in W^s(X(\omega))$. 
For $\nu$-almost all $\omega \in \Sigma_{k+1}$, $\sigma^n \omega \in S(\delta)$ for infinitely many positive integers $n$.
We may take $n$ large so that $\sigma^n \omega \in S(\delta)$ and 
$f^n_\omega (y) \in D_\delta (X( \sigma^n \omega)) \subset W^s (X(\sigma^n \omega))$.
By continuity of the diffeomorphisms $f_0, \ldots, f_k$, a small neighborhood of $y$ lies in $W^s (X(\omega))$. 


It remains to show that $W^s (X(\omega))$ is dense in $M$ for $\nu$-almost all $\omega\in \Sigma_{k+1}$.
We have that $(f^n_{\sigma^{-n} \omega})_* m$ converges to
$\delta_{X(\omega)}$, $\nu$-almost surely.
This implies convergence in measure, and since $\sigma$ leaves $\nu$ invariant,
also that $(f^n_{\omega})_* m $ converges to $\delta_{X(\sigma^n \omega)}$ in measure.
That is, for any $\varepsilon >0$,
\begin{align*}
\nu \{ \omega \in \Sigma_{k+1} \; \vert \; d_{\mathcal{M}} ((f^n_\omega)_* m , \delta_{X(\sigma^n \omega)})  > \varepsilon \} &\to 0
\end{align*}
as $n \to \infty$. Here, as before, $d_\mathcal{M}$ is a metric on $\mathcal{M}$ generating the weak star topology.
This in turn implies that for some subsequence $n_k \to \infty$,
\begin{equation}\label{e:subseq}
\nu \{ \omega\in \Sigma_{k+1} \; \vert \; \lim_{k\to\infty} d_\mathcal{M} ( (f^{n_k}_\omega)_*  m  , \delta_{X(\sigma^{n_k} \omega) }) = 0 \} = 1
\end{equation}
(see e.g. \cite[Theorem~II.10.5]{shi84}).

We combine this with the existence of stable sets around $X(\sigma^n \omega)$ to prove 
that $d_\mathcal{M} ((f^n_{\omega})_* m , \delta_{X  (\sigma^n \omega  )}) \to 0$ almost surely.
In more detail,  let
\begin{align*} 
\Gamma(\hat{\delta} , N ) &= \{ \omega \in \Sigma_{k+1} \; \vert \; d_\mathcal{M} ( (f^N_\omega)_* m , \delta_{X(\sigma^N \omega)}   ) < \hat{\delta}  \}.
\end{align*}
Now \eqref{e:subseq} implies that 
for any given $\varepsilon>0$, $\hat{\delta} >0$,  there is $N >0$ with 
\begin{align}\label{e:hatdelta}
\nu (\Gamma (\hat{\delta} ,N  ) ) &> 1 - \varepsilon.
\end{align}
A measure is close to a delta measure if most of the measure is in a small ball: 
for any $\varepsilon, \delta$ there is $\hat{\delta} >0$ so that $d_\mathcal{M} (\mu , \delta_x) < \hat{\delta}$ implies
$\mu ( D_\delta (x) ) > 1 - \varepsilon$.
So \eqref{e:hatdelta} gives that for any $\varepsilon>0, \delta >0$ there exists $N>0$ so that
\begin{align*}
\nu \{ \omega \in \Sigma_{k+1}  \; \vert \;  (f^N_\omega)_* m  ( D_\delta (X(\sigma^N\omega)) ) > 1 - \varepsilon   \} &> 1-\varepsilon. 
\end{align*}
With \eqref{e:S} we get that for all $\varepsilon >0$, there exists $\delta>0$ and $N >0$ so that 
 the set 
\begin{align*}
 T_\varepsilon &= \{ \omega \in \Sigma_{k+1}  \; \vert \; \text{for } n \ge N,  (f^n_\omega)_* m  ( D_\delta (X(\sigma^n\omega)) ) > 1 - \varepsilon 
  \} 
\end{align*}
satisfies $\nu (T_\varepsilon) > 1-\varepsilon$.

Let $U\subset \Sigma_{k+1}$ be the set of $\omega \in \Sigma_{k+1}$ with $\sigma^i \omega \in T_\varepsilon$  for 
infinitely many integers $i$  for each $\varepsilon$.
Note that $\nu (U) = 1$.
Suppose $\omega \in U$.
Take $y \in M$ and a small ball $B$ around it.  To prove that $W^s(X(\omega))$ is dense in $M$, we must show that $B$ contains  a point in $W^s (X(\omega))$. 
Note that $m (B) >0$ since $m$ has full support. For $\varepsilon >0$ small enough we have 
$m (B) > \varepsilon$. Therefore, for $\omega\in U$, there is $z \in B$ with $\sigma^i \omega \in T_\varepsilon$ and 
$f^{n+i} (z) \in   B_\delta (X(\sigma^{n+i}\omega)) \subset W^s (X(\sigma^{n+i} \omega))$, $n\ge N$.
\end{proof}

\subsection{Uniqueness of the stationary measure}

This section addresses the uniqueness of the stationary measure stated in item~\ref{i:ii} of Theorem~\ref{t:synch}.

\begin{lemma}\label{l:unique}
Assume the conditions of Theorem~\ref{t:synch}.
Then the stationary measure $m$ is the unique stationary measure. 
\end{lemma}

\begin{proof}
Let $m$ be a stationary measure with only negative fiber Lyapunov exponents and assume there exists a different
stationary measure $\hat{m}$. We may take $\hat{m}$ to be an ergodic  stationary measure.
By Proposition~\ref{p:furstenberg} there is a $F$-invariant measure $\hat{\mu}$ with marginal $\nu$ and conditional measures
$\hat{\mu}_\omega$ satisfying 
\begin{align}\label{e:furforhatmu}
\lim_{n\to \infty} (f^n_{\sigma^{-n} \omega})_* \hat{m} &= \hat{\mu}_\omega
\end{align}
for $\nu$-almost all $\omega \in \Sigma_{k+1}$.

Recall that for $\nu$-almost all $\omega \in \Sigma_{k+1}$,  $W^s (X(\omega))$ is open and dense (Lemma~\ref{l:opendense}).
We can therefore take $e_1$ and a subset $T \subset \Sigma_{k+1}$ with $\nu(T)>0$ so that for $\omega \in T$,
$W^s (X(\omega))$ contains a closed ball $B(\omega)$ of diameter $e_1$.
By Lemma~\ref{l:fullsupport}, $\hat{m}$ has full support and thus assigns positive measure to any open set $B \subset M$.
We can therefore take $e_2>0$ and decrease $T$ if necessary to find 
\begin{align}\label{e:e2}
\hat{m} (B(\omega)) &>  e_2
\end{align}
for $\omega \in T$.
By taking $T$ still smaller if needed, 
we may moreover assume that there are $e_3>0$ and $N>0$ so that
$f^n_{\sigma^{-n} \omega}  (B (\sigma^{-n} \omega)) $ is contained in
a ball of diameter $e_3$ around $X(\omega)$, 
if $n\ge N$  and $\sigma^{-n} \omega \in T$.
In particular $f^n_{\sigma^{-n} \omega}  (B (\sigma^{-n} (\omega)))$ converges to $X(\omega)$
if $\sigma^{-n} \omega \in T$ and $n \to \infty$.
For $\nu$-almost all $\omega \in \Sigma_{k+1}$, $\sigma^{-n} \omega \in T$ for infinitely many values of $n$.
For such $\omega$ we find by \eqref{e:e2} that
$(f^n_{\sigma^{-n} \omega})_*  \hat{m} (f^n_{\sigma^{-n} \omega}  (B (\sigma^{-n} (\omega))) > e_2$.
Since $f^n_{\sigma^{-n} \omega}  (B (\sigma^{-n} (\omega)))$ converges to $X(\omega)$ we get
that there is an accumulation point of $(f^n_{\sigma^{-n} \omega})_*  \hat{m}$
that assign positive measure to $ X(\omega)$.
%
%
By \eqref{e:furforhatmu}, 
\begin{align*}
\hat{\mu}_\omega  (X(\omega)) &>0.
\end{align*}
However, by ergodicity, $\hat{\mu} \ne \mu$ implies $\hat{\mu}_\omega (X(\omega)) = 0$ for $\nu$-almost all $\omega \in \Sigma_{k+1}$.
A contradiction has been derived and the lemma is proved.
\end{proof}

We proved Theorem~\ref{t:synch}.

\section{Iterated function systems on compact surfaces}\label{s:2}

On compact two-dimensional surfaces one obtains Theorem~\ref{t:synch} with iterated function systems generated by two diffeomorphisms.

\begin{proposition}\label{p:synchM2}
Let $M$ be a compact two-dimensional surface.
There is a $C^1$ open set of iterated function systems
generated by $C^2$ diffeomorphisms $f_0,f_{1}$ on $M$
with the following properties.
 
\begin{enumerate}[label=(\roman*)]
 \item the iterated function systems $\text{IFS}\, (\{ f_0,f_1\})$ and $\text{IFS}\, (\{f_0^{-1},f_1^{-1}\})$ are minimal;
 \item the iterated function system admits a unique  a stationary measure $m$ of full support;
 \item the iterated function system  has only negative fiber Lyapunov exponents; 
 \item for almost all $\omega \in \Sigma^+$ there is an open, dense set $W(\omega)\subset M$  so that 
  \begin{align*} 
   \lim_{n\to \infty} d( f^n_{\omega} (x) , f^n_{\omega} (y))
  =0
  \end{align*}
  for  $x,y \in W(\omega)$.  
\end{enumerate}
\end{proposition}

The proof of Theorem~\ref{t:synch} can be followed, with Lemma~\ref{l:m=0:gen} being replaced by
Lemma~\ref{l:m=0:M} below. This lemma uses that $\text{IFS}\, (\{f_0^{-1},f_1^{-1}\})$
is minimal. The construction can be done so that both $\text{IFS}\, (\{ f_0,f_1\})$ and $\text{IFS}\, (\{f_0^{-1},f_1^{-1}\})$ are minimal.

\begin{lemma}
\label{l:m=0}
A stationary measure $m$ is atom free.
\end{lemma}

\begin{proof}
Following \cite[Proposition~6]{1086.37026}, 
we claim that $m$ is atom free.
Take otherwise a point $p$ with maximal positive mass.
Then $f_i^{-1} (p)$, $i=0,1$, all have the same mass.
Taking further inverse images leads to an infinite set of points
(a finite set would contradict minimality $\text{IFS}\,\{f_0^{-1},f_1^{-1}\}$)  with the same positive mass, a contradiction.
\end{proof}


Recall that $f_0$ has an attracting fixed point $Q_0$ with $W^s (Q_0)$ open and dense in $M$. 
As before, $\Lambda_0 = M \setminus W^s (Q_0)$ is a stratification. 
For an open set of diffeomorphisms $f_i$, $i=0,1$,  $f_1^{-1} (\Lambda_0) $ is transverse to 
$\Lambda_0$. 

\begin{lemma}
\label{l:m=0:M}
Assume that $f_1^{-1} (\Lambda_0) $ is transverse to 
$\Lambda_0$.
A stationary measure $m$ then satisfies 
\begin{align*} 
m (\Lambda_0) &< 1/2. 
\end{align*}
\end{lemma}

\begin{proof}
Write $\alpha = m(\Lambda_0)$.
Since $m$ has full support, $\alpha <1$.
 Since $M$ is two-dimensional,  $\Lambda_0$  intersects $f_1^{-1}(\Lambda_0)$
in a set of dimension zero, if it intersects, so in a set of stationary measure zero by Lemma~\ref{l:m=0}.
Therefore $m(   \Lambda_0 \cup  f^{-1}_1 (\Lambda_0) ) =  m (\Lambda_0) +  m( f^{-1}_1 (\Lambda_0)) = 
\alpha + f_1 m(\Lambda_0)$, so that
$\alpha + f_1 m(\Lambda_0) <1$.

The measure $m$ being stationary implies
$\alpha = p_0 \alpha + p_1 (f_1)_* m(\Lambda_0)$.
So \[\alpha =   (f_1)_* m(\Lambda_0) \]
and $\alpha + (f_1)_* m(\Lambda_0) <1$ implies $\alpha< 1/2$.
\end{proof}


\subsection{An example on the torus}


Figure~\ref{f:synchonT2} gives a numerical demonstration of synchronization in a specific iterated function system on the two dimensional torus. 
Here we provide a robust synchronization result for small perturbations of this specific iterated function system,
partly to illustrate the results of this paper.

\begin{proposition}\label{p:example}
Let $\text{IFS}\,\{ f_0,f_1\}$ on the two dimensional torus $\mathbb{T}^2 = \mathbb{R}^2 / \mathbb{Z}^2$ be defined by
\begin{align*}
f_0(x,y) &= \left(x + \frac{1}{4\pi} \sin (2\pi x) , y + \frac{1}{4\pi} \sin (2\pi y)\right),
\\
f_1 (x,y) &= \left(x + \frac{1}{\pi} , y + \frac{1}{\sqrt{2}\pi}\right) \mod 1.
\end{align*}
Assume the diffeomorphisms $f_0,f_1$ are picked independently with positive probabilities $p_0, p_1 = 1-p_0$.

Then there is a $C^1$ neighborhood of $\text{IFS}\,\{ f_0,f_1\}$ so that
any iterated function system  $\text{IFS}\,\{ g_0,g_1\}$ with $C^2$ diffeomorphisms from it, 
is minimal and displays synchronization.
\end{proposition}

\begin{proof}
Observe that $\text{IFS}\,\{ f_0,f_1\}$ is a product of iterated function systems on the circle $\mathbb{T}$:
$\text{IFS}\,\{ f_{0,1},f_{1,1}\}$  with
\begin{align*}
f_{0,1}(x) = \left(x + \frac{1}{4\pi} \sin (2\pi x) \right),  \;
f_{1,1} (x) = \left(x + \frac{1}{\pi} \right) \mod 1.
\end{align*}
and   $\text{IFS}\,\{ f_{0,2},f_{1,2}\}$
with 
\begin{align*}
f_{0,2}(y) = \left(y + \frac{1}{4\pi} \sin (2\pi y)\right), \;
f_{1,2} (y) = \left( y + \frac{1}{\sqrt{2}\pi}\right) \mod 1.
\end{align*}
These iterated function systems on the circle are clearly minimal, possess a unique stationary measure $m_1,m_2$
and display synchronization \cite{1086.37026}.
Moreover,  the fiber Lyapunov exponents are negative \cite[Theorem 7.1]{klevol14}.


For any $\varepsilon>0$, $u \in \mathbb{T}$ near $0$, there exists $n \in \mathbb{N}$, so that 
$ f_{0,1} \circ f_{1,1}^n$ has an attracting fixed point within distance $\varepsilon$ of $1/2 + u$.
Similarly for $f_{0,2}, f_{1,2}$.
Because $f_1$ is a minimal diffeomorphism, we conclude that
for any $\varepsilon>0$, $(u,v) \in \mathbb{T}^2$ near $(0,0)$, there exists $n \in \mathbb{N}$, so that 
$ f_0 \circ f_1^n$ has an attracting fixed point within $\varepsilon$ of $(1/2 + u, 1/2 + v)$.

Let $B$ be a small ball around the attracting fixed point $(1/2,1/2)$ of $f_0$.
Then $f_0 (B)$ is strictly contained inside $B$.
Take vectors $(u_1,v_1),\ldots (u_l,v_l)$ so that $B \subset \cup_{i=1}^l  f_0(B) + (u_i,v_i)$.
With the above observation it is easily seen that there are numbers $n_1,\ldots,n_l$ so that
\begin{align}\label{e:B1} 
B &\subset f_0 (B) \cup \cup_{i=1}^l   f_0 \circ f_1^{n_i} (B). 
\end{align}
Note further that there is a finite number $N$ so that 
\begin{align}\label{e:B2}
\mathbb{T}^2 &\subset \cup_{i=0}^N f_1^i (B).
\end{align}
Since \eqref{e:B1} and \eqref{e:B2} are robust under small $C^1$ perturbations of $f_0,f_1$,
it follows that $\text{IFS}\,\{ f_0,f_1\}$ is robustly minimal on $\mathbb{T}^2$.

Now $\text{IFS}\,\{ f_0,f_1\}$ has a stationary measure $m = m_1 \times m_2$ with negative fiber Lyapunov exponents.
By Lemma~\ref{l:unique}, $m$ is the unique stationary measure for   $\text{IFS}\,\{ f_0,f_1\}$.
Since $m_1$ and $m_2$ are nonatomic \cite{1086.37026}, we find $m(W^s ((1/2,1/2)) =1 $ for
the basin of attraction $W^s ((1/2,1/2))$ of the
stable fixed point $(1/2,1/2)$ of $f_0$.

Consider now $C^2$ diffeomorphisms $g_0,g_1$ that are $C^1$ close to $f_0,f_1$.
Then by Lemma~\ref{l:c0}, any stationary measure for $\text{IFS}\,\{ g_0,g_1\}$ is close to $m$ in the weak star topology.
So $g_0$ has an attracting fixed point whose basin has stationary measure at least $1/2$.
It now follows from Theorem~\ref{t:synch} that $\text{IFS}\,\{ g_0,g_1\}$ has a unique stationary measure with negative fiber Lyapunov exponents,
and displays synchronization. 
\end{proof}

\def\cprime{$'$}


\begin{thebibliography}{10}


\bibitem{ant84}
V.A. Antonov.
\newblock Modeling of processes of cyclic evolution type. synchronization by a
  random signal.
\newblock {\em Vestnik Leningrad. Univ. Mat. Mekh. Astronom.} 2:67--76, 1984.


\bibitem{arn98}
L.~Arnold.
\newblock {\em Random dynamical systems}.
\newblock Springer Verlag, 1998.








\bibitem{bbd}
J.~Bochi, C. Bonatti, L.J. D\'{\i}az.
\newblock Robust vanishing of all Lyapunov exponents for iterated function systems.
\newblock {\em Math. Z.} 276:469--503, 2014.
  




%


%


\bibitem{cra90}
H.~Crauel.
\newblock Extremal exponents of random dynamical systems do not vanish.
\newblock {\em J. Dynam. Differential Equations} 2:245--291, 1990.


\bibitem{MR2358052}
B.~Deroin, V.A.~Kleptsyn,  A.~Navas.
\newblock Sur la dynamique unidimensionnelle en r{\'e}gularit{\'e}
 interm{\'e}diaire.
\newblock {\em Acta Math.} 199:199--262, 2007.
 
%










    


\bibitem{flagessch17a}
F. Flandoli, B. Gess, M. Scheutzow.
\newblock Synchronization by noise.
\newblock {\em Probab. Theory Related Fields} 168:511--556, 2017.

\bibitem{flagessch17b}
F. Flandoli, B. Gess, M. Scheutzow.
\newblock Synchronization by noise for order-preserving random dynamical systems.
\newblock {\em Ann. Probab.} 45:1325--1350, 2017.


\bibitem{ifs}
G.H.~Ghane, A.J.~Homburg,  A.~Sarizadeh.
\newblock ${C}^1$ robustly minimal iterated function systems.
\newblock {\em Stoch. Dyn.} 10:155--160, 2010.


\bibitem{ghahom16}
M. Gharaei, A.J. Homburg.
\newblock Random interval diffeomorphisms.
\newblock {\em Discrete Contin. Dyn. Syst. Ser. S} 10:241--272, 2017.



%




\bibitem{Hi}
M.W.~Hirsch.
\newblock {\em Differential topology}.
\newblock Springer Verlag,  1976.

 

\bibitem{hom}
A.J. Homburg.
\newblock Atomic disintegrations for partially hyperbolic diffeomorphisms.
\newblock {\em Proc. Amer. Math. Soc.} 145: 2981--2996, 2017.

\bibitem{homnas}
A.J.~Homburg, M.~Nassiri.
\newblock Robust minimality of iterated function systems with two generators.
\newblock {\em Ergod. Th.  Dyn. Systems} 34:1914--1929, 2014.



\bibitem{hut81}
J. Hutchinson. 
\newblock Fractals and self-similarity.
\newblock {\em Indiana Univ. Math. J.} 30:713--747, 1981.








%
%
%
\bibitem{kai93}
T.~Kaijser.
\newblock On stochastic perturbations of iterations of circle maps.
\newblock {\em Phys. D} 68:201--231, 1993.

\bibitem{1086.37026}
V.A. Kleptsyn, M.B. Nalskii.
\newblock {Contraction of orbits in random dynamical systems on the circle.}
\newblock {\em Funct. Anal. Appl.} 38:267--282, 2004.

\bibitem{klevol14}
V.A. Kleptsyn, D. Volk.
\newblock Physical measures for nonlinear random walks on interval. 
\newblock {\em Mosc. Math. J.} 14:339--365, 2014.

\bibitem{lej87}
Y.~{Le Jan}.
\newblock Equilibre statistique pour les produits de diff{\'e}omorphismes
  al{\'e}atoires ind{\'e}pendants.
\newblock {\em Ann. Inst. H. Poincar{\'e} Probab. Statist.} 23:111--120, 1987.


\bibitem{mal17}
D. Malicet.
\newblock Random Walks on Homeo($S^1$).
\newblock {\em Commun. Math. Phys.} 356:1083--1116, 2017.

\bibitem{MR889254}
R.~Ma{\~n}{\'e}.
\newblock {\em Ergodic theory and differentiable dynamics}.
\newblock Springer-Verlag,  1987.


 \bibitem{masyou02}
 N.~Masmoudi, L.-S.~Young.
 \newblock Ergodic theory of infinite dimensional systems with applications of dissipative parabolic PDEs. 
 \newblock {\em Commun. Math. Phys.} 227:461--481, 2002. 

%
%


\bibitem{pc90}
L. ~M.~Pecora, T.~L.~Carroll,
\newblock Synchronization in chaotic systems.
\newblock {\em Phys. Rev.  Lett.} 64:821--824, 1990. 


\bibitem{MR2068774}
Ya.B. Pesin.
\newblock {\em Lectures on partial hyperbolicity and stable ergodicity}.
\newblock Zurich Lectures in Advanced Mathematics. European Mathematical
  Society (EMS),  2004.

\bibitem{synch}
A.~Pikovsky, M.~Rosenblum, J.~Kurths.
\newblock {\em Synchronization. A universal concept in nonlinear sciences}. 
\newblock Cambridge University Press,  2001.


%




%
 \bibitem{shi84}
 A.N.~Shiryayev.
 \newblock {\em Probability}.
 \newblock  Springer Verlag, 1984. 
 

 \bibitem{MR1605989}
 J.~Stark.
 \newblock Invariant graphs for forced systems.
 \newblock {\em Phys. D} 109:163--179, 1997.

\bibitem{sta99}
J.~Stark.
\newblock Regularity of invariant graphs for forced systems.
\newblock {\em Ergod. Th. Dyn. Systems} 19:155--199, 1999. 


\bibitem {ste01}
D.~Steinsaltz.
\newblock  Random logistic maps and Lyapunov exponents.
\newblock {\em Indag. Math. (N.S.)} 12:557--584, 2001.



\bibitem{via14}
M. Viana.
\newblock {\em Lectures on Lyapunov exponents}.
\newblock Cambridge University Press, 2014.



 
 
\bibitem{MR2425065}
H.~Zmarrou, A.J. Homburg.
\newblock Dynamics and bifurcations of random circle diffeomorphisms.
\newblock {\em Discrete Contin. Dyn. Syst. Ser. B} 10:719--731, 2008.

\end{thebibliography}
\end{document}